\documentclass[11pt,a4paper]{article}

\usepackage{amsmath}
\usepackage{amsthm}
\usepackage{amsfonts}
\usepackage{color}
\usepackage{setspace}
\usepackage{mathrsfs}
\usepackage{graphicx}

\usepackage{calc}
\usepackage{tikz}

\newtheorem{theorem}{Theorem}
\newtheorem{conjecture}{Conjecture}
\newtheorem{lemma}{Lemma}
\newtheorem{proposition}{Proposition}

\newtheorem{case}{Case}
\newtheorem{observation}{Observation}
\newtheorem{subcase}{Subcase}[case]
\newtheorem{claim}{Claim}
\newtheorem{corollary}{Corollary}

\usetikzlibrary{decorations.markings}
\tikzstyle{vertex}=[circle, draw, inner sep=2pt, minimum size=6pt]

\tikzstyle{filledvertex}=[circle, draw, fill, inner sep=2pt, minimum size=6pt]
\newcommand{\vertex}{\node[vertex]}

\tikzstyle{directed}=[postaction={decorate,
	decoration={markings,mark=at position 0.5 with {\arrow{stealth}}}
}]

\usepackage{geometry}
\begin{document}
\begin{spacing}{1.2}

\title{Kernels by properly colored paths in arc-colored digraphs
\thanks{The first author is supported by NSFC (Nos. 11601430 and 11671320), NPU (No. 2016KY0101) and China Postdoctoral Science Foundation (No. 2016M590969);
the second author is supported by  JSPS KAKENHI (No. 15K04979);
and the third author is supported by NSFC (Nos. 11571135 and 11671320).
Part of this work was done while the first and the third author were visiting Yokohama City University and the hospitality was appreciated.}}

\author{\quad Yandong Bai $^{a,}$\thanks{Corresponding author. E-mail addresses: bai@nwpu.edu.cn (Y. Bai), shinya.fujita.ph.d@gmail.com (S. Fujita), sgzhang@nwpu.edu.cn (S. Zhang).},
\quad Shinya Fujita $^{b}$,
\quad Shenggui Zhang $^{a}$\\[2mm]
\small $^{a}$ Department of Applied Mathematics, Northwestern Polytechnical University, \\
\small Xi'an 710129, China\\
\small $^{b}$ International College of Arts and Sciences, Yokohama City University, \\
\small Yokohama 236-0027, Japan}

\date{\today}
\maketitle

\begin{abstract}
A {\em kernel by properly colored paths} of an arc-colored digraph $D$
is a set $S$ of vertices of $D$ such that
(i) no two vertices of $S$ are connected by a properly colored directed path in $D$,
and (ii) every vertex outside $S$ can reach $S$ by a properly colored directed path in $D$.
In this paper,
we conjecture that every arc-colored digraph with all cycles properly colored has such a kernel
and verify the conjecture for unicyclic digraphs, semi-complete digraphs and bipartite tournaments, respectively. 
Moreover, weaker conditions for the latter two classes of digraphs are given.

\smallskip
\end{abstract}

\section{Introduction}
All graphs (digraphs) considered in this paper are finite and simple,
i.e., without loops or multiple edges (arcs).
For terminology and notation not defined here,
we refer the reader to Bang-Jensen and Gutin \cite{BG2008}.

A path (cycle) in a digraph always means a {\em directed} path (cycle)
and a {\em $k$-cycle $C_{k}$} means a cycle of length $k$, where $k\geq 2$ is an integer.
For a digraph $D$,
define its {\em kernel} to be a set $S$ of vertices of $D$ such that
(i) no two vertices of $S$ are connected by an arc in $D$,
and (ii) every vertex outside $S$ can reach $S$ by an arc in $D$.
This notion was originally introduced by von Neumann and Morgenster \cite{VM1944} in 1944.
Since it has many applications in both cooperative games and logic (see \cite{Berge1977,Berge1984}),
its existence has been the focus of extensive study,
both from the algorithmic perspective and the sufficient condition perspective.
Among them, the following results are of special importance.
For more results on kernels,
we refer the reader to the survey paper \cite{BG2006} by Boros and Gurvich.

\begin{theorem}[Chv\'{a}tal \cite{Chvatal1973}]\label{thm: finding a kernel is npc}
It is NP-complete to recognize whether a digraph has a kernel or not.
\end{theorem}

\begin{theorem}[Richardson \cite{Richardson1953}, von Neumann and Morgenster \cite{VM1944}]
\label{thm: kernels in digraphs}
Let $D$ be a digraph.
Then the following statements hold:\\
$(i)$ if $D$ has no cycle, then $D$ has a unique kernel;\\
$(ii)$ if $D$ has no odd cycle, then $D$ has at least one kernel;\\
$(iii)$ if $D$ has no even cycle, then $D$ has at most one kernel.
\end{theorem}

An arc $uv\in A(D)$ is called {\em symmetrical} if $vu\in A(D)$.
For a cycle $(u_{0},u_{1},\ldots,u_{k-1},u_{0})$,
we call two arcs $u_{i}u_{i+2}$ and $u_{i+1}u_{i+3}$ {\em crossing consecutive}, where addition is modulo $k$.
The following theorem has been proved.

\begin{theorem}[Duchet \cite{Duchet1980}, Duchet and Meyniel \cite{DM1983}, Galeana-S\'{a}nchez and Neumann-Lara\cite{GN1984}]
\label{thm: kernels under some constraits for odd cycles}
A digraph $D$ has a kernel if one of the following conditions holds:\\
$(i)$ each cycle has a symmetrical arc;\\
$(ii)$ each odd cycle has two crossing consecutive arcs;\\
$(iii)$ each odd cycle has two chords whose heads are adjacent vertices.
\end{theorem}

It is worth noting that if we replace Condition (ii) in the definition of kernels by every vertex outside $S$ can reach $S$ by an arc or a path of length 2,
then such a vertex subset,
named {\em quasi-kernel},
always exists.
This was proved by Chv\'{a}tal and Lov\'{a}sz \cite{CL1974} in 1974.
Jacob and Meyniel \cite{JM1996} furthermore showed in 1996 that every digraph has either a kernel or three quasi-kernels.
For more results on quasi-kernels,
see \cite{BC2012,GPR1991,HH2008}.

Let $D$ be a digraph and $m$ a positive integer.
Call $D$ an {\em $m$-colored digraph} if its arcs are colored with at most $m$ colors.
Denote by $c(uv)$ the color assigned to the arc $uv$.
A subdigraph $H$ of an arc-colored digraph $D$ is called {\em monochromatic}
if all arcs of $H$ receive the same color,
and is called {\em rainbow}
if any two arcs of $H$ receive two distinct colors.
Define a {\em kernel by monochromatic paths}
(or an {\em MP-kernel} for short) of an arc-colored digraph $D$
to be a set $S$ of vertices of $D$ such that
(i) no two vertices of $S$ are connected by a monochromatic path in $D$,
and (ii) each vertex outside $S$ can reach $S$ by a monochromatic path in $D$.

The concept of MP-kernels in an arc-colored digraph was introduced by Sands, Sauer and Woodrow \cite{SSW1982} in 1982.
They showed that
every 2-colored digraph has an MP-kernel.
In particular, as a corollary,
they showed that every 2-colored tournament has a one-vertex MP-kernel.
Here note that each MP-kernel of an arc-colored tournament consists of one vertex.
They also proposed the problem that
whether a 3-colored tournament with no rainbow triangles has a one-vertex MP-kernel.
This problem still remains open
and has attracted many authors to investigate sufficient conditions for the existence of MP-kernels in arc-colored tournaments.
Shen \cite{Shen1988} showed in 1988 that for $m\geq 3$
every $m$-colored tournament with no rainbow triangles and no rainbow transitive triangles has a one-vertex MP-kernel,
and also showed that the condition ``with no rainbow triangles and no rainbow transitive triangles" cannot be improved for $m\geq 5$.
In 2004, Galeana-S\'{a}nchez and Rojas-Monroythe \cite{GR2004} showed,
by constructing a family of counterexamples,
that the condition of Shen cannot be improved for $m=4$, either.
Galeana-S\'{a}nchez \cite{Galeana1996} showed in 1996 that
every arc-colored tournament such that
the arcs, with at most one exception, of each cycle of length at most four are assigned the same color has a one-vertex MP-kernel.
Besides, Galeana-S\'{a}nchez and Rojas-Monroythe \cite{GR2004285} showed in 2004 that
every arc-colored bipartite tournament with all 4-cycles monochromatic has an MP-kernel.
For more results on MP-kernels,
we refer to the survey paper \cite{Galeana1998} by Galeana-S\'{a}nchez.

A subdigraph $H$ of an arc-colored digraph $D$ is called {\em properly colored}
if any two consecutive arcs of $H$ receive distinct colors.
Define a {\em kernel by properly colored paths}
(or a {\em PCP-kernel} for short) of an arc-colored digraph $D$
to be a set $S$ of vertices of $D$ such that
(i) no two vertices of $S$ are connected by a properly colored path in $D$,
and (ii) each vertex outside $S$ can reach $S$ by a properly colored path in $D$.

By the definitions of kernels, MP-kernels and PCP-kernels,
one can see in some sense
that both MP-kernels and PCP-kernels generalize the concept of kernels in digraphs.

\begin{observation}\label{observation: three kernels}
Let $D=(V(D),A(D))$ be a digraph.
Then the following three statements are equivalent.\\
$(i)$ $D$ has a kernel;\\
$(ii)$ $|A(D)|$-colored $D$ has an MP-kernel;\\
$(iii)$ $1$-colored $D$ has a PCP-kernel.
\end{observation}

In this paper we concentrate on providing some sufficient conditions for the existence PCP-kernels in arc-colored digraphs.
For convenience,
we write ``PC path" for ``properly colored path" in the following.
Define the {\em closure} $\mathscr{C}(D)$ of an arc-colored digraph $D$
to be a digraph with vertex set $V(\mathscr{C}(D))=V(D)$
and arc set $A(\mathscr{C}(D))=\{uv:$ there is a PC $(u,v)$-path in $D\}$.
It is not difficult to see that the following simple (but useful) result holds.

\begin{observation}\label{observation: kernels in D and its closure}
An arc-colored digraph $D$ has a PCP-kernel if and only if $\mathscr{C}(D)$ has a kernel.
\end{observation}

\section{Main results}

We first consider the computational complexity of finding a PCP-kernel in an arc-colored digraph.

\begin{proposition}\label{proposition: finding a pcp-kernel is np-hard}
It is NP-hard to recognize whether an arc-colored digraph has a PCP-kernel or not.
\end{proposition}

\begin{proof}
Let $D$ be a digraph and $V^{*}$ a set of vertices with $V^{*}\cap V(D)=\emptyset$.
Let $D'$ be the digraph with $V(D')=V(D)\cup V^{*}$ and $A(D')=A(D)\cup \{uv:~u\in V^{*}, v\in V(D)\}$,
i.e., adding a set $V^{*}$ of new vertices to $D$ together with all possible arcs from $V^{*}$ to $V(D)$.
We can always choose a $V^{*}$ with $|\{uv:~u\in V^{*}, v\in V(D)\}|\geq m$.
Color $D'$ by using $m$ colors in such a way that the subdigraph $D$ is monochromatic
and the arc set $\{uv:~u\in V^{*}, v\in V(D)\}$ is $m$-colored.
It is not difficult to see that the $m$-colored $D'$ has a PCP-kernel if and only if $D$ has a kernel.
By Theorem \ref{thm: finding a kernel is npc} the computational complexity of the latter problem is NP-complete.
The desired result then follows directly.
\end{proof}

Now we present the following result.

\begin{proposition}\label{proposition: pcp-kernels in digraphs}
An arc-colored digraph $D$ has a PCP-kernel if one of the following conditions holds:\\
$(i)$ $D$ has no cycle;\\
$(ii)$ the coloring of $D$ is proper (consecutive arcs receive distinct colors);\\
$(iii)$ $D$ is properly-connected (each vertex can reach all other vertices by a PC path).
\end{proposition}

\begin{proof}
Note that $\mathscr{C}(D)$ has a cycle if and only if $D$ has a cycle.
The statement (i) therefore follows directly from Theorem \ref{thm: kernels in digraphs} (i) and Observation \ref{observation: kernels in D and its closure}.
Assume that the coloring of $D$ is proper.
If $D$ is strongly connected,
then each vertex forms a PCP-kernel.
If $D$ is not strongly connected,
then the set of sinks is a PCP-kernel.
If $D$ is properly-connected,
then by the definition of PCP-kernels each vertex forms a PCP-kernel.
\end{proof}

By Proposition \ref{proposition: pcp-kernels in digraphs} (i),
every arc-colored digraph containing no cycle has a PCP-kernel.
It is natural to ask what is the analogous answer for a digraph $D$ containing cycles.
For the simplest case, i.e., $D$ is a cycle, we get the following result.

\begin{theorem}\label{thm: pcp-kernels in cycles}
An arc-colored cycle has a PCP-kernel if and only if it is not a monochromatic odd cycle.
\end{theorem}

Call a digraph {\em unicyclic} if it contains exactly one cycle.
Note that every cycle is unicyclic.
For general arc-colored unicyclic digraphs,
furthermore, for general digraphs containing cycles,
a number of examples
(see for example the arc-colored digraphs
in Figures \ref{figure: a 2-colored digraph with no pcp-kernel} and \ref{figure: a 3-colored bipartite tournament with no PCP-kernel})
show that additional conditions  are needed to guarantee the existence of PCP-kernels.
But what kind of conditions do we need?
By Proposition \ref{proposition: pcp-kernels in digraphs} (ii) and (iii),
if the coloring is proper or ``close" to proper (roughly speaking),
then it has a PCP-kernel.
By Proposition \ref{proposition: pcp-kernels in digraphs} (i),
the existence of cycles influences the existence of PCP-kernels.
This yields a natural question to ask whether the condition ``all cycles are properly colored" suffices or not.
Based on this consideration, we propose the following conjecture.

\begin{conjecture}\label{conj: pcp-kernels in digraphs}
Every arc-colored digraph with all cycles properly colored has a PCP-kernel.
\end{conjecture}

\noindent
\textbf{Remark 1.}
If Conjecture \ref{conj: pcp-kernels in digraphs} is true,
then it is best possible in view of the two arc-colored digraphs in Figure \ref{figure: a 2-colored digraph with no pcp-kernel},
in which solid arcs, dotted arcs and dashed arcs represent arcs colored by three distinct colors respectively.
It is not difficult to check that neither of them has a PCP-kernel.
For any even integer $n\geq 6$ (resp. odd integer $n\geq 7$),
the sharpness of Conjecture \ref{conj: pcp-kernels in digraphs}
can be shown by replacing the path $(v_{6},v_{1},v_{2})$ (resp. $(u_{9},u_{1},u_{2})$) of the left digraph (resp. the right digraph)
by a monochromatic path of length $n-4$ (resp. length $n-7$) using the color assigned to the previous short path.
One can check that neither of the two new constructed digraphs has a PCP-kernel.

\begin{figure}[ht]
\begin{center}
\begin{tikzpicture}
\tikzstyle{vertex}=[circle,inner sep=2pt, minimum size=0.1pt]


\vertex (a)[fill] at (-5,0)[label=left:$v_{1}$]{};
\vertex (b)[fill] at (-4,1.5)[label=above:$v_{2}$]{};
\vertex (c)[fill] at (-2,1.5)[label=above:$v_{3}$]{};
\vertex (d)[fill] at (-1,0)[label=right:$v_{4}$]{};
\vertex (e)[fill] at (-2,-1.5)[label=below:$v_{5}$]{};
\vertex (f)[fill] at (-4,-1.5)[label=below:$v_{6}$]{};

\vertex (g)[fill] at (1,0)[label=left:$u_{1}$]{};
\vertex (h)[fill] at (2,1.5)[label=above:$u_{2}$]{};
\vertex (i)[fill] at (4,1.5)[label=above:$u_{3}$]{};
\vertex (j)[fill] at (5,1)[label=right:$u_{4}$]{};
\vertex (k)[fill] at (4,0.5)[label=left:$u_{5}$]{};
\vertex (l)[fill] at (4,-0.5)[label=left:$u_{6}$]{};
\vertex (m)[fill] at (5,-1)[label=right:$u_{7}$]{};
\vertex (n)[fill] at (4,-1.5)[label=below:$u_{8}$]{};
\vertex (o)[fill] at (2,-1.5)[label=below:$u_{9}$]{};


\draw [directed,line width=1pt] (a)--(b);
\draw [directed,line width=1pt] (b)--(c);
\draw [directed,line width=1pt] (c)--(d);
\draw [directed,line width=1pt] (e)--(f);
\draw [directed,line width=1pt] (f)--(a);

\draw [directed,dotted,line width=1pt] (c)--(e);
\draw [directed,dotted,line width=1pt] (e)--(d);

\draw [directed,line width=1.2pt] (g)--(h);
\draw [directed,line width=1pt] (h)--(i);
\draw [directed,line width=1pt] (i)--(j);
\draw [directed,line width=1pt] (k)--(l);
\draw [directed,line width=1pt] (l)--(m);
\draw [directed,line width=1pt] (n)--(o);
\draw [directed,line width=1pt] (o)--(g);

\draw [directed,dotted,line width=1pt] (i)--(k);
\draw [directed,dotted,line width=1pt] (k)--(j);

\draw [directed,dashed,line width=1pt] (l)--(n);
\draw [directed,dashed,line width=1pt] (n)--(m);

\end{tikzpicture}
\end{center}
\caption{Two arc-colored digraphs with no PCP-kernels.}
\label{figure: a 2-colored digraph with no pcp-kernel}
\end{figure}
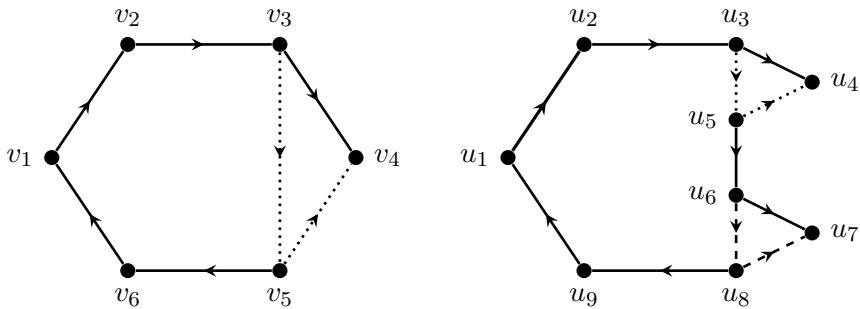

A digraph $D$ is {\em semi-complete} if for every two vertices there exists at least one arc between them.
A {\em tournament} ({\em bipartite tournament})
is an orientation of a complete graph (complete bipartite graph).
Note that each tournament is semi-complete.
Theorem \ref{thm: pcp-kernels in cycles} shows that Conjecture \ref{conj: pcp-kernels in digraphs} holds for cycles.
We will also show that Conjecture \ref{conj: pcp-kernels in digraphs} holds for general unicyclic digraphs, semi-complete digraphs and bipartite tournaments.
In fact, for the latter two classes of digraphs,
weaker conditions have been obtained, respectively.

\begin{theorem}\label{thm: pcp-kernels in unicyclic digraphs}
Every arc-colored unicyclic digraph with the unique cycle properly colored has a PCP-kernel.
\end{theorem}

\noindent
\textbf{Remark 2.}
We see from the two unicyclic arc-colored digraphs in Figure \ref{figure: a 2-colored digraph with no pcp-kernel}
that the condition ``the unique cycle is properly colored" cannot be dropped in Theorem \ref{thm: pcp-kernels in unicyclic digraphs}.

Note that every two vertices in a semi-complete digraph are adjacent
and thus every PCP-kernel in such a digraph consists of one vertex.
We obtain the following result whose proof idea is similar to that in \cite{Shen1988}.

\begin{theorem}\label{thm: PCP kernels in semi-complete digraphs}
Every arc-colored semi-complete digraph with no monochromatic triangles has a vertex $v$
such that all other vertices can reach $v$ by a PC path of length at most 3.
\end{theorem}

\begin{corollary}\label{corollary: PCP kernels in semi-complete digraphs}
Every arc-colored semi-complete digraph with no monochromatic triangles has a PCP-kernel.
\end{corollary}

\noindent
\textbf{Remark 3.}
The condition ``with no monochromatic triangles" in Theorem \ref{thm: PCP kernels in semi-complete digraphs} and Corollary \ref{corollary: PCP kernels in semi-complete digraphs} cannot be dropped.
Recall that every tournament is semi-complete
and one can verify that the 2-colored tournament shown in Figure \ref{figure: a 2-colored tournament} 
has no PCP-kernels and no vertex defined in Theorem \ref{thm: PCP kernels in semi-complete digraphs},
in which solid arcs and dotted arcs represent arcs colored by two distinct colors, respectively.
Larger $m$-colored tournaments containing no PCP-kernel for general $m$
can be constructed by adding new vertices together with new colors to the new added arcs
such that $T^{*}$ has no outneighbors in the set of new added vertices.

\begin{figure}[ht]
\begin{center}
\begin{tikzpicture}[x=1.2cm, y=1.2cm, every edge/.style=
        {
        draw,line width=1pt,
        postaction={decorate,
                     decoration={markings,mark=at position 0.5 with {\arrow{stealth}}}
                   }
        }
]
\tikzstyle{vertex}=[circle,inner sep=2pt, minimum size=0.1pt]

    \vertex (a)[fill] at (0,2) [label=above:{$v_{1}$}]{};

    \vertex (b)[fill] at (0,0) [label=left:{$v_{2}$}]{};

    \vertex (c)[fill] at (-1.732,-1) [label=left:{$v_{3}$}]{};

    \vertex (d)[fill] at (1.732,-1) [label=right:{$v_{4}$}]{};
    \path (b) edge[dotted] (c);
    \path (c) edge[dotted] (d);
    \path (d) edge[dotted] (b);
    \path (a) edge[] (b);
    \path (a) edge[] (c);
    \path (a) edge[] (d);
\end{tikzpicture}
\end{center}
\caption{A 2-colored tournament with no PCP-kernels.}
\label{figure: a 2-colored tournament}
\end{figure}
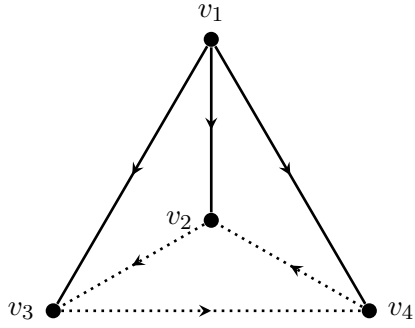

\begin{theorem}\label{thm: pcp-kernels in bipartite tournaments}
Every arc-colored bipartite tournament $D=(X,Y;A)$ with
$(i)$ all $4$-cycles and $6$-cycles properly colored, or
$(ii)$ $\min\{|X|,|Y|\}\leq 2$,
has a PCP-kernel.
\end{theorem}

\noindent
\textbf{Remark 4.}
The conditions in Theorem \ref{thm: pcp-kernels in bipartite tournaments} cannot be dropped
in view of the 3-colored bipartite tournament $D_{6}=(X,Y;A)$ shown in Figure \ref{figure: a 3-colored bipartite tournament with no PCP-kernel},
in which solid arcs, dotted arcs and dashed arcs represent arcs colored by three distinct colors, respectively.
One can see that $\min\{|X|,|Y|\}=3$
and $D_{6}$ contains neither PC 4-cycles nor PC 6-cycles.
One can also see that the closure $\mathscr{C}(D_{6})$ of $D_{6}$ is semi-complete,
in which the new added arcs are represented by thick dashed lines.
Note that a semi-complete digraph has a kernel if and only if it has a vertex $v$ such that
all other vertices can reach $v$ by an arc.
One can see that $\mathscr{C}(D_{6})$ does not contain such a vertex,
so by Observation \ref{observation: kernels in D and its closure} we get that $D_{6}$ has no PCP-kernels.
Furthermore,
we can construct infinite family of bipartite tournaments which can show that the conditions in Theorem \ref{thm: pcp-kernels in bipartite tournaments} cannot be dropped.
Let $D_{n-6}$ be an arbitrary $m$-colored bipartite tournament with $n>6$.
Define $D$ to be the union of $D_{6}$ and $D_{n-6}$ as follows:
take all possible arcs between $D_{n-6}$ and $D_{6}$ going from $D_{n-6}$ to $D_{6}$ and denote this set of arcs by $A^{*}$,
let $V(D)=V(D_{6})\cup V(D_{n-6})$ and $A(D)=A(D_{6})\cup A(D_{n-6})\cup A^{*}$,
let the colors on $D_{n-6}$ and $D_{6}$ remain the same
and let the coloring of $A^{*}$ be arbitrary.
Then $D$ has no PCP-kernel
since the proposition that $D$ has a PCP-kernel implies that $D_{6}$ has a PCP-kernel.

\begin{figure}[ht]
\begin{center}
\begin{tikzpicture}[x=0.6cm, y=0.6cm, every edge/.style=
        {
        draw,line width=1pt,
        postaction={decorate,
                     decoration={markings,mark=at position 0.3 with {\arrow{stealth}}}
                   }
        }
]


    \vertex (x1)[fill] at (-10,2) [label=above:{$x_{1}$}]{};
    \vertex (x2)[fill] at (-6,2) [label=above:{$x_{2}$}]{};
    \vertex (x3)[fill] at (-2,2) [label=above:{$x_{3}$}]{};

    \vertex (y1) at (-10,-2) [label=below:{$y_{1}$}]{};
    \vertex (y2) at (-6,-2) [label=below:{$y_{2}$}]{};
    \vertex (y3) at (-2,-2) [label=below:{$y_{3}$}]{};


    \path (x1) edge[] (y1);
    \path (y1) edge[] (x2);
    \path (y2) edge[] (x1);

    \path (x2) edge[dotted] (y2);
    \path (y2) edge[dotted] (x3);
    \path (y3) edge[dotted] (x2);

    \path (x3) edge[dashed] (y3);
    \path (y1) edge[dashed] (x3);
    \path (y3) edge[dashed] (x1);


    \vertex (x1')[fill] at (2,2) [label=above:{$x_{1}$}]{};
    \vertex (x2')[fill] at (6,2) [label=above:{$x_{2}$}]{};
    \vertex (x3')[fill] at (10,2) [label=above:{$x_{3}$}]{};

    \vertex (y1') at (2,-2) [label=below:{$y_{1}$}]{};
    \vertex (y2') at (6,-2) [label=below:{$y_{2}$}]{};
    \vertex (y3') at (10,-2) [label=below:{$y_{3}$}]{};

    \path (x1') edge[] (y1');
    \path (y1') edge[] (x2');
    \path (y2') edge[] (x1');

    \path (x2') edge[dotted] (y2');
    \path (y2') edge[dotted] (x3');
    \path (y3') edge[dotted] (x2');

    \path (x3') edge[dashed] (y3');
    \path (y1') edge[dashed] (x3');
    \path (y3') edge[dashed] (x1');

    \path (x3') edge[dashed,line width=2pt] (x2');
    \path (x2') edge[dashed,line width=2pt] (x1');
    \path (x1') edge[dashed,line width=2pt,bend left=30] (x3');

    \path (y1') edge[dashed,line width=2pt] (y2');
    \path (y2') edge[dashed,line width=2pt] (y3');
    \path (y3') edge[dashed,line width=2pt,bend left=30] (y1');
    \path (y1') edge[dashed,line width=2pt,bend left=30] (x1');
    \path (y2') edge[dashed,line width=2pt,bend left=30] (x2');
    \path (y3') edge[dashed,line width=2pt,bend right=30] (x3');

\end{tikzpicture}
\end{center}
\caption{A 3-colored bipartite tournament $D_{6}$ and its closure $\mathscr{C}(D_{6})$.}
\label{figure: a 3-colored bipartite tournament with no PCP-kernel}
\end{figure}
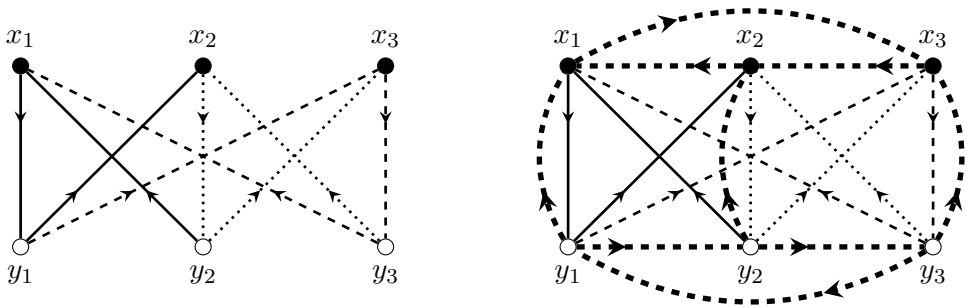

In the rest of the paper,
we always use $H_{1}-H_{2}$ to denote $H_{1}-V(H_{2})$ for two digraphs $H_{1}$ and $H_{2}$;
if $H_{2}$ consists of a single vertex $v$,
then we denote $H_{1}-\{v\}$ by $H_{1}-v$.
For two vertices $u$ and $v$,
if $uv$ is an arc then we say $u$ {\em dominates} $v$ and sometimes write $u\rightarrow v$.

\section{Proofs of Theorem \ref{thm: pcp-kernels in cycles} and Theorem \ref{thm: pcp-kernels in unicyclic digraphs}}

\begin{proof}[\bf{Proof of Theorem \ref{thm: pcp-kernels in cycles}}]
The necessity of Theorem \ref{thm: pcp-kernels in cycles} follows from the fact that each odd cycle has no kernel.
For the sufficiency,
it is equivalent to show that
(i) every arc-colored odd cycle with at least two colors has a PCP-kernel
and (ii) every arc-colored even cycle has a PCP-kernel.
We prove the result by constructing such a kernel $S$.

Let $C=(v_{0},\ldots,v_{n-1},v_{0})$ be an arc-colored cycle
and assume w.l.o.g. that the vertices are located in a clockwise direction.
If $C$ is an monochromatic even cycle,
then we can let $S=\{v_{0},v_{2},\ldots,v_{n-2}\}$.
Now assume that $C$ is an arc-colored cycle with at least two colors.
If the coloring is proper,
then clearly each vertex forms a PCP-kernel.
Now assume that the coloring is not proper and assume w.l.o.g. that
$P_{1}=(v_{n_{1}},v_{n_{1}+1},\ldots,v_{n'_{1}}=v_{n-1})$ is a monochromatic path of maximum length (which is at least two).
Put $v_{n'_{1}-2},v_{n'_{1}-4},\ldots,v_{n'_{1}-2t_{1}}$ into $S$,
where $t_{1}$ is the largest integer such that $n'_{1}-2t_{1}\geq n_{1}$.
Here, note that since $C$ is neither a monochromatic odd cycle nor a PC cycle,
we have $n_{1}\in \{0,1,\ldots,n-3\}$ and $n'_{1}-2t_{1}=n_{1}$ or $n_{1}+1$.
Afterwards, we consider, in a counter-clockwise direction,
the first appeared maximal monochromatic path of length at least two in $C-P_{1}$,
say $P_{2}=(v_{n_{2}},v_{n_{2}+1},\ldots,v_{n'_{2}})$.
Now put $v_{n'_{2}-2},v_{n'_{2}-4},\ldots,v_{n'_{2}-2t_{2}}$ into $S$,
where $t_{2}$ is the largest integer such that $n'_{2}-2t_{2}\geq n_{2}$.
Continue this procedure until there is no monochromatic path of length at least two
and let $P_{r}=(v_{n_{r}},v_{n_{r}+1},\ldots,v_{n'_{r}})$ be the last appeared maximal monochromatic path of length at least two.
It follows that
$$
S=\bigcup_{i=1}^{r}\{v_{n'_{i}-2},v_{n'_{i}-4},\ldots,v_{n'_{i}-2t_{i}}\},~~\overline{S}=V(C)\backslash S=\overline{S}'\cup \overline{S}'',
$$
where
$$
\overline{S}'=\bigcup_{i=1}^{r}\{v_{n'_{i}-3},v_{n'_{i}-5},\ldots,v_{n'_{i}-2t_{i}+1}\},
$$
$$
\overline{S}''=\bigcup_{i=1}^{r}\{v_{n'_{i}-2t_{i}-1},v_{n'_{i}-2t_{i}-2},\ldots,v_{n'_{i+1}},v_{n'_{i+1}-1}\},
$$
$n_{r+1}=n_{1}$, $n'_{r+1}=n'_{1}=n-1$ and addition is modulo $n$.
It is not difficult to check that no two vertices of $S$ are connected by a PC path in $C$.
For each $1\leq i\leq r$,
one can also verify that each vertex in $\{v_{n'_{i}-3},v_{n'_{i}-5},\ldots,v_{n'_{i}-2t_{i}+1}\}$
can reach some vertex in $\{v_{n'_{i}-2},v_{n'_{i}-4},\ldots,v_{n'_{i}-2t_{i}+2}\}$ by a PC path of length one,
and each vertex in $\{v_{n'_{i}-2t_{i}-1}$,
$v_{n'_{i}-2t_{i}-2},\ldots,v_{n'_{i+1}},v_{n'_{i+1}-1}\}$ can reach $v_{n'_{i}-2t_{i}}$ by a PC path;
in other words,
every vertex outside $S$ can reach $S$ by a PC path in $C$.
Therefore, the set $S$ is a PCP-kernel of $C$.
\end{proof}

\begin{proof}[\bf{Proof of Theorem \ref{thm: pcp-kernels in unicyclic digraphs}}]
Let $D$ be an arc-colored unicyclic digraph with a PC cycle $C$.
Note that the cycle $C$ must be an induced cycle since otherwise two cycles will appear.
Note also that each vertex of $C$ forms a PCP-kernel of $C$.
If $D$ is strongly connected,
then $D$ is a cycle and the desired result follows directly.
Now assume that $D$ is not strongly connected.
Then there exist strongly connected components $D_{1},\ldots,D_{k}$, $k\geq 2$, of $D$
such that there is no arc from $D_{i}$ to $D_{j}$ for any $i>j$.
Let $D_{i}$ be the component containing the cycle $C$.
One can see that $D_{i}=C$.
One can also see that each $D_{j}\neq D_{i}$ is a single vertex, since otherwise another cycle will appear.
We distinguish two cases and show the result by constructing a PCP-kernel $S$.

If $i=k$,
then let $v$ be an arbitrary vertex of $D_{k}=C$ and we put $v$ into $S$.
Let $j_{1}\in \{1,\ldots,k-1\}$ be the largest integer such that there is no PC $(D_{j_{1}},v)$-path.
Put $D_{j_{1}}$ into $S$.
Let $j_{2}\in \{1,\ldots,j_{1}-1\}$ be the largest integer such that there is no PC $(D_{j_{2}},\{v,D_{j_{1}}\})$-path.
Put $D_{j_{2}}$ into $S$.
Continue this procedure until all the remaining vertices in $V(C)\backslash S$ can reach $S$ by a PC path.
Let $D_{j_{r}}$ be the last vertex putting into $S$.
The terminal vertex set $S=\{v,D_{j_{1}},\ldots,D_{j_{r}}\}$ is clearly a PCP-kernel.

If $i\neq k$,
then $D$ contains at least one sink and we put all sinks, say $v_{1},\ldots,v_{p}$, into $S$.
By similar procedure above we can put, step by step, the vertices $D_{j_{1}},\ldots,D_{j_{t}}$ with $j_{t}>i$ into $S$.
Let $U\subseteq V(D_{i})$ be the set of vertices which cannot reach the current $S=\{v_{1},\ldots,v_{p},D_{j_{1}},\ldots,D_{j_{t}}\}$ by a PC path.
If $U\neq \emptyset$,
then put an arbitrary vertex of $U$ (instead of all vertices of $U$) into $S$ and continue the procedure.
If $U=\emptyset$,
then $j_{t+1}<i$ and we can use the same procedure above to get a PCP-kernel $S$.
\end{proof}

\section{Proof of Theorem \ref{thm: PCP kernels in semi-complete digraphs}}

For convenience,
in this proof,
call a vertex $v$ {\em good}
if all other vertices can reach $v$ by a PC path of length at most 3.
One can see that it suffices to consider the tournament case.
Let $T$ be an $m$-colored tournament,
where $m$ is a positive integer.
For $m=1$,
note that each monochromatic tournament with no monochromatic triangles is transitive,
then the unique sink is a good vertex.
So we may assume that $m\geq 2$ and $T$ is an arc-colored tournament with at least two colors.
We prove the result by induction on $|V(T)|$.

Since each arc-colored transitive triangle and each non-monochromatic triangle has a good vertex,
the result holds for $|V(T)|=3$.
Now assume that $T$ is a minimum counterexample with $|V(T)|=k\geq 4$.
It follows that each $m$-colored tournament with no monochromatic triangles and
with order less than $k$ has a good vertex.
So for each vertex $v$ of $T$
the subtournament $T-v$ has a good vertex.
Denote by $v^*$ the good vertex of $T-v$ corresponding to the given coloring of $T$.
Then $v^{*}\rightarrow v$,
since otherwise $v^{*}$ is a good vertex of $T$.
For two distinct vertices $u$ and $v$,
we claim that $u^*\neq v^*$.
If not,
then by the definition of $u^{*}$
there exist a PC $(v,u^{*})$-path in $T-u$ and a PC $(u,u^{*})$-path in $T-v$.
It follows immediately that there exist a PC $(v,u^{*})$-path and a PC $(u,u^{*})$-path in $T$.
Thus, $u^{*}$ is a good vertex of $T$, a contradiction.

Now consider the subdigraph $H$ induced on the arc set $\{v^{*}v:v\in V(T)\}$.
Since each vertex of $H$ has both indegree and outdegree one,
then $H$ consists of vertex-disjoint cycles.
If $H$ has at least two cycles,
then by induction hypothesis the induced subtournament on each cycle has a good vertex,
which is obviously a good vertex of $T$, a contradiction.
So $H$ consists of one cycle.

Let $H=(v_{0},v_{1},\ldots,v_{n-1},v_{0})$.
By the choices of the arcs,
there exists no PC $(v_{i},v_{i-1})$-path of length at most 3 in $T$,
addition is modulo $n$ in this proof.
Consider the three vertices $v_{0},v_{1},v_{2}$,
if $v_{2}\rightarrow v_{0}$,
then since there exists no monochromatic triangle
we have either $(v_{2},v_{0},v_{1})$ is PC $(v_{2},v_{1})$-path of length 2
or $(v_{1},v_{2},v_{0})$ is a PC $(v_{1},v_{0})$-path of length 2, a contradiction.
So $v_{0}\rightarrow v_{2}$.
In fact,
one can see from the simple proof that $v_{i}\rightarrow v_{i+2}$ for any $v_{i}\in V(T)$.

Let $s$ be the minimum integer such that
$v_{s}\rightarrow v_{0}$
and $v_{0}\rightarrow v_{i}$ for any $i\leq s-1$.
Such an integer $s$ exists by the fact that $v_{n-1}\rightarrow v_{0}$.
We may assume that $v_{i}\rightarrow v_{j}$ for any $1\leq i<j\leq s$.
Since there exists no PC $(v_{s},v_{s-1})$-path of length 2,
we have $c(v_{s}v_{0})=c(v_{0}v_{s-1})$, say $c(v_{s}v_{0})=c(v_{0}v_{s-1})=1$.
By assumption, there exists no monochromatic triangle,
we have $c(v_{s-1}v_{s})\neq c(v_{s}v_{0})$.
Since there exists no PC $(v_{s-1},v_{s-2})$-path of length 3,
we have $c(v_{0}v_{s-2})=c(v_{s}v_{0})=1$.
Since $(v_{0},v_{s-2},v_{s},v_{0})$ is not a monochromatic triangle,
we have $c(v_{s-2}v_{s})\neq 1$.
Similarly,
we can show that $c(v_{0}v_{i})=1$ and $c(v_{i}v_{s})\neq 1$ for any $1\leq i\leq s-1$.
This implies that $c(v_{1}v_{s})\neq c(v_{s}v_{0})=1$ and a PC $(v_{1},v_{0})$-path $(v_{1},v_{s},v_{0})$ of length 2 appears, a contradiction.

\section{Proof of Theorem \ref{thm: pcp-kernels in bipartite tournaments}}

\begin{proof}[\bf{Proof of Theorem \ref{thm: pcp-kernels in bipartite tournaments} (i)}]
For the 1-colored case,
by Observation \ref{observation: three kernels},
it suffices to consider the existence of a kernel.
We claim that either $X$ or $Y$ is a kernel.
If $X$ is not a kernel,
then there exists $y\in Y$ such that each vertex of $X$ is an inneighbor of $y$,
implying that $Y$ is a kernel.
So every 1-colored bipartite tournament has a PCP-kernel (not necessary to satisfy the required condition).
In the following we assume $m\geq 2$
and consider PCP-kernels in $m$-colored bipartite tournaments with at least two colors.

We write $u\sim v$ if $u\rightarrow v$ or $v\rightarrow u$.
It is not difficult to verify,
see also in \cite{GR2004285},
that the following lemma holds.
We need to keep in mind of this lemma in the forthcoming proof.

\begin{lemma}[Galeana-S\'{a}nchez and Rojas-Monroy \cite{GR2004285}]\label{lemma: walks in bipartite tournaments}
Let $D$ be an arc-colored bipartite tournament.
Then \\
$(i)$ for each directed walk $(u_{0},u_{1},\ldots,u_{k})$ in $D$ we have $u_{i}\sim u_{j}$ iff $j-i\equiv 1$ (mod 2);\\
$(ii)$ every closed directed walk of length at most 6 is a cycle in $D$.
\end{lemma}

For two vertices $u$ and $v$ in $D$,
denote by $dist(u,v)$ the distance from $u$ to $v$.
The following lemma will play a key role in the proof.

\begin{lemma}\label{lemma: any path to a 2-path}
If there exists a PC $(u,v)$-path
but exists no PC $(v,u)$-path in $D$,
then $dist(u,v)\leq 2$.
\end{lemma}

\begin{proof}
Let $P=(u_{0},u_{1},\ldots,u_{k})$ be a shortest PC $(u,v)$-path, where $u=u_{0}$ and $v=u_{k}$.
The result holds clearly for $k\leq 2$.
Now let $k\geq 3$
and assume the opposite that each $(u,v)$-path has length at least 3.

\setcounter{claim}{0}
\begin{claim}\label{claim: uk dominates the first k-2 vertices}
There exists no arc from $\{u_{0},u_{1},\ldots,u_{k-3}\}$ to $u_{k}$.
\end{claim}

\begin{proof}
The statement holds directly for $k\leq 4$.
Assume that $k\geq 5$.
Let $i^{*}=\min \{i:u_{i}\rightarrow u_{k},0\leq i\leq k-3\}$.
Then $i^{*}\in \{u_{2},u_{3},\ldots,u_{k-3}\}$ and $u_{k}\rightarrow u_{i^{*}-2}$.
Now $(u_{k},u_{i^{*}-2},u_{i^{*}-1}$,
$u_{i^{*}},u_{k})$ is a 4-cycle and by assumption it is properly colored.
So $c(u_{i^{*}-1}u_{i^{*}})\neq c(u_{i^{*}}u_{k})$
and $u_{0}Pu_{i^{*}}u_{k}$ is a PC $(u,v)$-path of length less than $k$, a contradiction.
\end{proof}

\begin{claim}
There exists $i\in \{0,1,\ldots,k-3\}$ such that $u_{i}\rightarrow u_{i+3}$.
\end{claim}

\begin{proof}
Assume the opposite that $u_{i+3}\rightarrow u_{i}$ for each $i\in \{0,1,\ldots,k-3\}$.
If $k$ is odd,
then $u_{0}\sim u_{k}$ and
either there exists a $(u,v)$-path of length 1
or there exists a PC $(v,u)$-path of length 1, a contradiction.
So $k$ is even.
Recall that $k\geq 3$.
If $k=4$,
then $(u_{0},u_{1},u_{2},u_{3},u_{0})$ and $(u_{1},u_{2},u_{3},u_{4},u_{1})$ are PC 4-cycles.
So $(u_{4},u_{1},u_{2},u_{3},u_{0})$ is a PC $(v,u)$-path, a contradiction.
If $k=6$,
then $u_{5}\rightarrow u_{0}$
since otherwise $(u_{0},u_{5},u_{6})$ is a $(u,v)$-path of length 2.
Now $(u_{3},u_{4},u_{5},u_{6},u_{3})$ is a PC 4-cycle
and $u_{0}Pu_{5}u_{0}$ is a PC 6-cycle.
Thus, $u_{6}u_{3}Pu_{5}u_{0}$ is a PC $(v,u)$-path, a contradiction.
If $k=8$,
then $u_{7}\rightarrow u_{0}$ and $u_{8}\rightarrow u_{1}$
since otherwise either $(u_{0},u_{7},u_{8})$ or $(u_{0},u_{1},u_{8})$ is a $(u,v)$-path of length 2.
Besides, we have $u_{5}\rightarrow u_{0}$
since otherwise $(u_{0},u_{5},u_{6},u_{7},u_{0})$
and $(u_{5},u_{6},u_{7},u_{8},u_{5})$ are PC 4-cycles
and $(u_{8},u_{5},u_{6},u_{7},u_{0})$ is a PC $(v,u)$-path.
We also can show that $u_{8}\rightarrow u_{3}$.
If not,
then $(u_{3},u_{8},u_{1},u_{2},u_{3})$ is a PC 4-cycle
and $(u_{0},u_{1},u_{2},u_{3},u_{8})$ is a PC $(u,v)$-path of length less than $k$, a contradiction.
Then there exist two PC 6-cycles
$(u_{0},u_{1},u_{2},u_{3},u_{4},u_{5},u_{0})$ and $(u_{3},u_{4},u_{5},u_{6},u_{7},u_{8},u_{3})$.
It follows that $u_{8}u_{3}Pu_{5}u_{0}$ is a PC $(v,u)$-path, a contradiction.
So from now on assume that $k\geq 10$.

We claim first that $u_{k}\rightarrow u_{k-5}$.
If not,
then $(u_{k},u_{k-3},u_{k-6},u_{k-5},u_{k})$ is a PC 4-cycle
and thus $u_{0}Pu_{k-5}u_{k}$ is a PC $(u,v)$-path of length less than $k$, a contradiction.
We also claim that $u_{k-3}\rightarrow u_{k-8}$.
If not,
then since $(u_{k-9},u_{k-8},u_{k-3},u_{k-6},u_{k-9})$
and $(u_{k-8},u_{k-3},u_{k-2},u_{k-1},u_{k},u_{k-5},u_{k-8})$ are PC cycles
we have $c(u_{k-9}u_{k-8})\neq c(u_{k-8}u_{k-3})$ and $c(u_{k-8}u_{k-3})\neq c(u_{k-3}u_{k-2})$.
It follows that $u_{0}Pu_{k-8}u_{k-3}Pu_{k}$ is a PC $(u,v)$-path of length less than $k$, a contradiction.

Recall that $u_{i+3}\rightarrow u_{i}$ for each $i\in \{0,1,\ldots,k-3\}$
and all 4-cycles and 6-cycles are properly colored.
Thus,
$$
u_{k}u_{k-5}Pu_{k-3}u_{k-8}u_{k-7}u_{k-10}\cdots u_{k-2i}u_{k-2i-1}u_{k-2i+2}\cdots u_{0}
$$
is a PC $(v,u)$-path,
contradicting the assumption in Lemma \ref{lemma: any path to a 2-path}.
\end{proof}

Let $i$ be the minimum integer in $\{0,1,\ldots,k-3\}$ such that $u_{i}\rightarrow u_{i+3}$
and let $j^{*}=\max \{j:u_{i}\rightarrow u_{j},i+3\leq j\leq k\}$.
By Claim \ref{claim: uk dominates the first k-2 vertices},
we have $j^{*}\neq k$.
If $j^{*}=k-1$,
then $i\neq 0$;
otherwise, $(u_{0},u_{k-1},u_{k})$ is a $(u_{0},u_{k})$-path of length 2.
By Claim \ref{claim: uk dominates the first k-2 vertices},
we also have $u_{k}\rightarrow u_{i-1}$.
Since $(u_{i-1},u_{i},u_{k-1},u_{k},u_{i-1})$ is a PC 4-cycle,
we get that $u_{0}Pu_{i}u_{k-1}u_{k}$ is a PC $(u_{0},u_{k})$-path of length less than $k$, a contradiction.
So we have $j^{*}\leq k-2$.

By the choice of $j^{*}$,
we have $u_{j^{*}+2}\rightarrow u_{i}$ and $(u_{i},u_{j^{*}},u_{j^{*}+1},u_{j^{*}+2},u_{i})$ is a PC 4-cycle.
Hence $c(u_{i}u_{j^{*}})\neq c(u_{j^{*}}u_{j^{*}+1})$.
If $i=0$,
then $u_{0}u_{j^{*}}Pu_{k}$ is a PC $(u_{0},u_{k})$-path of length less than $k$, a contradiction.
So $i\geq 1$.

By the minimality of $i$ we have $u_{i+2}\rightarrow u_{i-1}$.
Since $u_{i}\sim u_{j^{*}}$,
we have $u_{i+2}\sim u_{j^{*}+2}$.
If $u_{j^{*}+2}\rightarrow u_{i+2}$,
then $(u_{j^{*}+2},u_{i+2},u_{i-1},u_{i},u_{j^{*}},u_{j^{*}+1},u_{j^{*}+2})$ is a PC 6-cycle
and $u_{0}Pu_{i}u_{j^{*}}Pu_{k}$ is a PC $(u_{0},u_{k})$-path of length less than $k$, a contradiction.
So $u_{i+2}\rightarrow u_{j^{*}+2}$.

If $j^{*}\leq k-4$,
then by the choice of $j^{*}$ we have $u_{j^{*}+4}\rightarrow u_{i}$.
Now $(u_{i},u_{i+1},u_{i+2},u_{j^{*}+2}$,
$u_{j^{*}+3},u_{j^{*}+4},u_{i})$ is a PC 6-cycle
and $u_{0}Pu_{i+2}u_{j^{*}+2}Pu_{k}$ is a PC $(u_{0},u_{k})$-path of length less than $k$, a contradiction.
So $j^{*}\in \{k-2,k-3\}$.
If $j^{*}=k-2$,
then $(u_{i},u_{i+1},u_{i+2},u_{j^{*}+2},u_{i})$ is a PC 4-cycle
and $u_{0}Pu_{i+2}u_{j^{*}+2}$ is a PC $(u_{0},u_{k})$-path of length less than $k$, a contradiction.
So $j^{*}=k-3$ and $u_{j^{*}+3}=u_{k}$.

Now we claim that $u_{i+1}\rightarrow u_{k}$.
If not,
then $(u_{i+1},u_{i+2},u_{j^{*}+2},u_{k},u_{i+1})$ is a PC 4-cycle
and $u_{0}Pu_{i+2}u_{j^{*}+2}u_{k}$ is a PC $(u_{0},u_{k})$-path of length less than $k$, a contradiction.
We may also claim that $u_{i-1}\rightarrow u_{k}$.
If not,
then $(u_{i-1},u_{i},u_{i+1},u_{k},u_{i-1})$ is a PC 4-cycle
and $u_{0}Pu_{i+1}u_{k}$ is a PC $(u_{0},u_{k})$-path of length less than $k$, a contradiction.
Similarly, we can show that $u_{i-s}\rightarrow u_{k}$ for any odd $s$ with $1\leq s\leq i$.
Clearly, there will be a $(u,v)$-path of length at most 2.

The proof of Lemma \ref{lemma: any path to a 2-path} is complete.
\end{proof}

In view of Theorem \ref{thm: kernels under some constraits for odd cycles} (i),
it suffices to show that every cycle of $\mathscr{C}(D)$ has a symmetrical arc.
Assume the opposite that there exists a cycle $C$ in $\mathscr{C}(D)$ containing no symmetrical arc
and denote it by
$$
C=(u_{0},u_{1},\ldots, u_{l},u_{0}).
$$
We will get a contradiction by showing that $C$ has a symmetrical arc.
Here we distinguish two cases.

\setcounter{case}{0}
\begin{case}
$l=2$.
\end{case}

Since a bipartite tournament contains no odd cycle,
there exists an arc of $C$ in $A(\mathscr{C}(D))\backslash A(D)$, say $u_{0}u_{1}$.
By Lemma \ref{lemma: any path to a 2-path},
there exists a $(u_{0},u_{1})$-path of length 2 in $D$, say $(u_{0},x_{0},u_{1})$.

If $u_{1}u_{2},u_{2}u_{0} \in A(D)$,
then $(u_{0},x_{0},u_{1},u_{2},u_{0})$ is a (properly colored) 4-cycle in $D$ and $u_{1}u_{0}\in \mathscr{C}(D)$.
This implies that $C$ has a symmetrical arc $u_{0}u_{1}$, a contradiction.

If $|\{u_{1}u_{2},u_{2}u_{0}\}\cap A(D)|=1$,
then by Lemma \ref{lemma: any path to a 2-path} and Lemma \ref{lemma: walks in bipartite tournaments} (ii)
there will be a 5-cycle which contradicts the well-known fact that a bipartite tournament contains no odd cycle.

Now let $u_{1}u_{2},u_{2}u_{0} \notin A(D)$.
Then by Lemma \ref{lemma: any path to a 2-path},
there exist a $(u_{1},u_{2})$-path of length 2 and a $(u_{2},u_{0})$-path of length 2 in $D$,
say $(u_{1},x_{1},u_{2})$ and $(u_{2},x_{2},u_{0})$.
By Lemma \ref{lemma: walks in bipartite tournaments} (ii) and our assumption
we get that $(u_{0},x_{0},u_{1},x_{1},u_{2},x_{2},u_{0})$ is a PC 6-cycle.
This implies that each arc in $C$ is a symmetrical arc, a contradiction.

\begin{case}
$l\geq 3$.
\end{case}

In view of Lemma \ref{lemma: any path to a 2-path},
there exists a $(u_{i},u_{i+1})$-path of length at most 2 for each $0\leq i\leq l$ in $D$, where $u_{l+1}=u_{0}$.
Let $P_{i}$ be the shortest $(u_{i},u_{i+1})$-path in $D$
and let $C^{*}=\cup_{i=0}^{l}P_{i}$.
Then $C^{*}$ is a closed directed walk in $D$.
For convenience, denote this closed walk by
$$
C^{*}=(x_{0},x_{1},\ldots, x_{s},x_{0}),
$$
where $x_{0}=u_{0}$ and $s\geq l$.

If $x_{3}x_{0}\in A(D)$,
then $(x_{0},x_{1},x_{2},x_{3},x_{0})$ is a PC 4-cycle
and $x_{1}x_{0},x_{2}x_{0}\in A(\mathscr{C}(D))$.
Note that either $x_{0}x_{1}\in A(C)$ or $x_{0}x_{2}\in A(C)$.
This implies that $C$ has a symmetrical arc, a contradiction.
Similarly, if $x_{0}x_{s-2}\in A(D)$,
then we can show that either $x_{s}x_{0}$ or $x_{s-1}x_{0}$ is a symmetrical arc of $C$, a contradiction.

Now assume that $x_{0}x_{3},x_{s-2}x_{0}\in A(D)$.
Let $i$ be the minimum integer such that $x_{0}x_{i},x_{i+2}x_{0}\in A(D)$.
Then $(x_{0},x_{i},x_{i+1},x_{i+2},x_{0})$ is a PC 4-cycle in $D$
and $x_{i+1}x_{i},x_{i+2}x_{i}\in A(\mathscr{C}(D))$.
If $x_{i}\in V(C)$,
then $\{x_{i}x_{i+1},x_{i}x_{i+2}\}\cap A(C)\neq \emptyset$
and thus either $x_{i+1}x_{i}$ or $x_{i+2}x_{i}$ is a symmetrical arc in $C$, a contradiction.
So $x_{i}\notin V(C)$ and $x_{i-1}x_{i+1}\in A(C)$.
By the choice of $i$,
we have $x_{0}x_{i-2}\in A(D)$.
Then $(u_{0},u_{i-2},u_{i-1},u_{i},u_{i+1},u_{i+2},u_{0})$ is a PC 6-cycle
and there exists a PC $(x_{i+1},x_{i-1})$-path.
So $x_{i-1}x_{i+1}$ is a symmetrical arc in $C$, a contradiction.
\end{proof}

\begin{proof}[\bf{Proof of Theorem \ref{thm: pcp-kernels in bipartite tournaments} (ii)}]
If $\min \{|X|,|Y|\}=1$,
then $D$ has no cycle and the result follows from Proposition \ref{proposition: pcp-kernels in digraphs} (i).
So we can assume w.l.o.g. that $|X|=\min \{|X|,|Y|\}=2$.
By contradiction, suppose the opposite that $D$ has no PCP-kernel.
By Proposition \ref{proposition: pcp-kernels in digraphs} we can assume that $D$ has a cycle.
It is not difficult to check that
if $y\in Y$ is a source then $D$ has a PCP-kernel
if and only if $D-y$ has a PCP-kernel.
So we assume also that $D$ has no source in $Y$.
Let $X=\{x_{1},x_{2}\}$ and let
\begin{center}
$Y_{0}=\{y\in Y:~x_{1}\rightarrow y,~x_{2}\rightarrow y\},$
\end{center}
\begin{center}
$Y_{1}=\{y\in Y\backslash Y_{0}:~there~exists~a~PC~(y,Y_{0})$-path in $D$$\},$
\end{center}
\begin{center}
$Y_{2}=Y\backslash (Y_{0}\cup Y_{1})$.
\end{center}

If $Y_{2}=\emptyset$,
then $Y_{0}$ is a PCP-kernel.
So we assume that $Y_{2}\neq \emptyset$.
Two vertices $v_{1}$ and $v_{2}$ are called {\em contractible}
if for any vertices $u$ and $w$
we have $v_{1}\rightarrow u$ iff $v_{2}\rightarrow u$,
$w\rightarrow v_{1}$ iff $w\rightarrow v_{2}$,
and $c(v_{1}u)=c(v_{2}u),c(wv_{1})=c(wv_{2})$ whenever $v_{1}u,v_{2}u,wv_{1},wv_{2}\in A(D)$.
Recall that all digraphs we consider here are simple,
that is, contain no loops.
So there exists no arc between any two contractible vertices.
We now show the following claim.

\begin{lemma}\label{lemma: no two contractible vertices}
Let $v_{1},v_{2}$ be two contractible vertices in an arc-colored digraph $D'$.
Then $D'$ has a PCP-kernel if and only if $D'-v_{2}$ has a PCP-kernel.
\end{lemma}

\begin{proof}
For the necessity,
let $S$ be a PCP-kernel of $D'$.
If $\{v_{1},v_{2}\}\subseteq S$,
then by the definition of contractible vertices $S\backslash v_{2}$ is a PCP-kernel of $D'-v_{2}$.
If $v_{2}\in S$ and $v_{1}\notin S$,
then $S\cup \{v_{1}\}$ is a PCP-kernel of $D'-v_{2}$.
If $\{v_{1},v_{2}\}\cap S=\emptyset$,
then $S$ is also a PCP-kernel of $D'-v_{2}$.
For the sufficiency,
let $S'$ be a PCP-kernel of $D'-v_{2}$.
If $v_{1}\notin S'$,
then $S'$ is a PCP-kernel of $D'$.
Now assume that $v_{1}\in S'$.
If there exists a PC $(v_{2},v_{1})$-path,
then $S'$ is a PCP-kernel of $D'$.
Otherwise,
there exists no PC $(v_{1},v_{2})$-path
and $S'\cup \{v_{2}\}$ is a PCP-kernel of $D'$.
\end{proof}

Now we assume that $D$ does not contain two contractible vertices and distinguish two cases in the following.

\setcounter{case}{0}
\begin{case}
$Y_{0}\neq \emptyset$.
\end{case}

Since $D$ has no source in $Y$,
each vertex in $Y\backslash Y_{0}$ has one outneighbor and one inneighbor in $\{x_{1},x_{2}\}$.
For a vertex $x\in X$ which has at least one inneighbor in $Y_{2}$,
since there exists no PC path from $Y_{2}$ to $Y_{0}$,
we have $c(xy'_{0})=c(xy''_{0})$ for any $y'_{0},y''_{0}\in Y_{0}$;
otherwise,
for each $y\in Y_{2}$ with $y\rightarrow x$
there exists $y_{0}\in Y_{0}$ with $c(yx)\neq c(xy_{0})$,
which yields a PC path $(y,x,y_{0})$ from $Y_{2}$ to $Y_{0}$, a contradiction.
For convenience,
denote by $c(xY_{0})$ the common color assigned for the arcs from $x$ to $Y_{0}$.
By the definition of $Y_{2}$,
the following claim holds.

\setcounter{claim}{0}
\begin{claim}\label{claim: two vertices in Y2}
For two vertices $y',y''\in Y_{2}$,
if $\{y',y''\}\rightarrow x$ for some $x\in X$,
then $c(y'x)=c(y''x)=c(xY_{0})$.
\end{claim}

Let $S'$ be a maximal subset of $Y_{2}$ such that
no two vertices of $S'$ are connected by a PC path.
If $S'=Y_{2}$,
then $Y_{0}\cup S'$ is a PCP-kernel.
Assume that $S'\neq Y_{2}$.
Let
\begin{center}
$R=\{y\in Y_{2}\backslash S'$: there~exists~no~PC~$(y,S')$-path~in~$D$$\}$.
\end{center}
If $R=\emptyset$,
then $Y_{0}\cup S'$ is a PCP-kernel.
So assume that $R\neq \emptyset$.
Let $r$ be an arbitrary vertex in $R$.
Then by the definitions of $S'$ and $R$,
there exists a PC $(s',r)$-path for some $s'\in S'$.

\begin{claim}
Every PC $(s',r)$-path has length 2.
\end{claim}

\begin{proof}
By contradiction,
assume w.l.o.g. that there exists a PC $(s',r)$-path of length 4,
say $(s',x_{1},y,x_{2},r)$, where $y\in Y\backslash Y_{0}$.
Since $s'\notin Y_{1}$,
we have $c(s'x_{1})=c(x_{1}Y_{0})$ and $c(yx_{2})=c(x_{2}Y_{0})$.
We show that there exists a PC $(z,r)$-path for any $z\in Y_{2}-\{s',r\}$.
If $z=y$,
then $(z,x_{2},r)$ is a desired path.
Now let $z\neq y$.
Since $z\notin Y_{0}$,
we have either $z\rightarrow x_{1}$ or $z\rightarrow x_{2}$.
If $z\rightarrow x_{1}$,
then since $z\in Y_{2}$ we have $c(zx_{1})=c(x_{1}Y_{0})=c(s'x_{1})$
and $(z,x_{1},y,x_{2},r)$ is a desired path.
If $z\rightarrow x_{2}$,
then similarly $c(zx_{2})=c(x_{2}Y_{0})=c(yx_{2})$
and $(z,x_{2},r)$ is a desired path.
It follows that $Y_{0}\cup \{r\}$ is a PCP-kernel, a contradiction.
\end{proof}

Now we can assume w.l.o.g. that
$(s',x_{1},r)$ is a PC $(s',r)$-path.
Remark that
$c(s'x_{1})\neq c(x_{1}r)$
and, by Claim \ref{claim: two vertices in Y2},
each vertex $y\in Y_{2}$ with $y\rightarrow x_{1}$ can reach $r$ by a PC path $(y,x_{1},r)$.
Let
\begin{center}
$Q=\{y\in Y_{2}\backslash r:~x_{1}\rightarrow y\}.$
\end{center}
If $Q=\emptyset$,
then $Y_{0}\cup \{r\}$ is a PCP-kernel, a contradiction.
So assume that $Q\neq \emptyset$.

\begin{claim}\label{claim: no pc path from r to Q}
There exists no PC $(r,Q)$-path.
\end{claim}

\begin{proof}
Assume the opposite that there exists a PC $(r,Q)$-path,
say $(r,x_{2},y,x_{1},q)$,
for some $q\in Q$.
Then $c(yx_{1})=c(x_{1}Y_{0})$
since otherwise $(r,x_{2},y,x_{1},y_{0})$
is a PC $(r,y_{0})$-path for each $y_{0}\in Y_{0}$,
contradicting that $r\in Y_{2}$.
Now we show that $Y_{0}\cup \{q\}$ is a PCP-kernel.
Since $Q\cup \{r\}\subseteq Y_{2}$,
we have $c(qx_{2})=c(rx_{2})=c(x_{2}Y_{0})$ for each $q\in Q$.
So $(q',x_{2},y,x_{1},q)$ is a PC $(q',q)$ path for each $q'\in Q\backslash q$.
For each $y'\in Y_{2}\backslash Q$,
note that $y'\rightarrow x_{1}$,
since $y'\in Y_{2}$,
we have $c(y'x_{1})=c(x_{1}Y_{0})=c(yx_{1})$.
Then $(y',x_{1},q)$ is a PC $(y',q)$-path.
It therefore follows that $Y_{0}\cup \{q\}$ is a PCP-kernel.
\end{proof}

\begin{claim}\label{claim: no PC path connecting 2 vertices of Q}
There exists no PC path connecting two vertices of $Q$.
\end{claim}

\begin{proof}
By symmetry, assume that $(q',x_{2},y,x_{1},q'')$ is a PC path for some two vertices $q',q''\in Q$.
Note that $y\neq r$,
otherwise, there exists a PC $(r,Q)$-path $(y,x_{1},q'')$, contradicting Claim \ref{claim: no pc path from r to Q}.
Since $c(rx_{2})=c(q'x_{2})$,
we get that $(r,x_{2},y,x_{1},q'')$ a PC $(r,q'')$-path, a contradiction.
\end{proof}

Let $Q'\subseteq Q$ be the set of vertices which cannot reach $r$ by a PC path.
By Claims \ref{claim: no pc path from r to Q} and \ref{claim: no PC path connecting 2 vertices of Q},
no two vertices of $Q'\cup \{r\}$ are connected by a PC path.
It follows that $Y_{0}\cup Q'\cup \{r\}$ is a PCP-kernel, a contradiction.

\begin{case}
$Y_{0}=\emptyset$.
\end{case}

Recall that $D$ has no source in $Y$.
By the assumption we have that
every vertex in $Y$ has one outneighbor and one inneighbor in $\{x_{1},x_{2}\}$.
Let
$$
Y'=\{y\in Y:~x_{1}\rightarrow y,~y\rightarrow x_{2}\},~~Y''=Y\backslash Y'=\{y\in Y:~x_{2}\rightarrow y,~y\rightarrow x_{1}\},
$$
$$
Y^{*}=\{y\in Y':~c(x_{1}y)\neq c(yx_{2})\},~~Y^{**}=\{y\in Y'':~c(x_{2}y)\neq c(yx_{1})\}.
$$
In the following proof we need to keep in mind that
each vertex in $Y'$ can reach $x_{2}$ by a PC path
and each vertex in $Y''$ can reach $x_{1}$ by a PC path.
If $Y^{*}\cup Y^{**}=\emptyset$,
i.e., there exists no PC path connecting $x_{1}$ and $x_{2}$,
then clearly $\{x_{1},x_{2}\}$ is a PCP-kernel.
Now let $Y^{*}\cup Y^{**}\neq \emptyset$
and assume w.o.l.g. that $Y^{*}\neq \emptyset$.
If there exist $y'_{1},y'_{2}\in Y^{*}$ with $c(x_{1}y'_{1})\neq c(x_{1}y'_{2})$,
then since each vertex in $Y''$ can reach $x_{2}$ by a PC path passing through either $\{x_{1},y'_{1}\}$ or $\{x_{1},y'_{2}\}$
we get that $\{x_{2}\}$ is a PCP-kernel.
So we can assume that $c(x_{1}y')=\alpha$ for each $y'\in Y^{*}$.
Let
$$
Y'_{\alpha}=\{y\in Y':~c(x_{1}y)=c(yx_{2})=\alpha\},~~Y''_{\alpha}=\{y\in Y'':~c(x_{2}y)=c(yx_{1})=\alpha\}.
$$
We now claim that $Y^{**}\neq \emptyset$.
Assume the opposite that $Y^{**}=\emptyset$.
If $Y''_{\alpha}=\emptyset$,
then since each vertex in $Y''$ can reach $x_{2}$ by a PC path passing through $x_{1}$ and an arbitrary vertex in $Y^{*}$
we get that $\{x_{2}\}$ is a PCP-kernel.
If $Y''_{\alpha}\neq \emptyset$ and $Y'_{\alpha}=\emptyset$,
then since each vertex in $Y'$ can reach $Y''_{\alpha}$ by a PC path passing through $x_{2}$,
and each vertex in $Y''\backslash Y''_{\alpha}$ can reach $Y''_{\alpha}$ by a PC path passing through $x_{1}$ together with an arbitrary vertex in $Y^{*}$ and $x_{2}$,
we can get that $Y''_{\alpha}$ is a PCP-kernel.
If $Y''_{\alpha}\neq \emptyset$ and $Y'_{\alpha}\neq \emptyset$,
then by a similar analysis and the observation that no two vertices of $Y'_{\alpha}\cup Y''_{\alpha}$ are connected by a PC path
we have that $Y'_{\alpha}\cup Y''_{\alpha}$ is a PCP-kernel.
So $Y^{**}\neq \emptyset$.

If there exist $y''_{1},y''_{2}\in Y^{**}$ with $c(x_{2}y''_{1})\neq c(x_{2}y''_{2})$,
then similar to the analysis for $Y^{*}$ we have that $\{x_{1}\}$ is a PCP-kernel.
Thus, we can assume that $c(x_{2}y'')=\beta$ for each $y''\in Y^{**}$.
For the sake of a better presentation,
define the following vertex sets,
see also in Figure \ref{figure: an arc-colored BT with no sink and no source}
in which a vertex encircled may represent a set of vertices,
and solid arcs, dotted arcs, dashed arcs represent respectively the arcs colored by $\alpha$, $\beta$ and a color not in $\{\alpha,\beta\}$.
$$
Y'_{\beta}=\{y\in Y':~c(x_{1}y)=c(yx_{2})=\beta\},~~Y''_{\beta}=\{y\in Y'':~c(x_{2}y)=c(yx_{1})=\beta\},
$$
$$
Y'_{\gamma}=\{y\in Y':~c(x_{1}y)=c(yx_{2})\notin \{\alpha,\beta\}\},~~Y''_{\gamma}=\{y\in Y'':~c(x_{2}y)=c(yx_{1})\notin \{\alpha,\beta\}\},
$$
$$
Y'_{\alpha\beta}=\{y\in Y':~c(x_{1}y)=\alpha,~c(yx_{2})=\beta\},~~Y'_{\alpha\gamma}=\{y\in Y':~c(x_{1}y)=\alpha,~c(yx_{2})\notin \{\alpha,\beta\}\},
$$
$$
Y''_{\beta\alpha}=\{y\in Y'':~c(x_{2}y)=\beta,~c(yx_{1})=\alpha\},~~Y'_{\beta\gamma}=\{y\in Y'':~c(x_{2}y)=\beta,~c(yx_{1})\notin \{\alpha,\beta\}\}.
$$

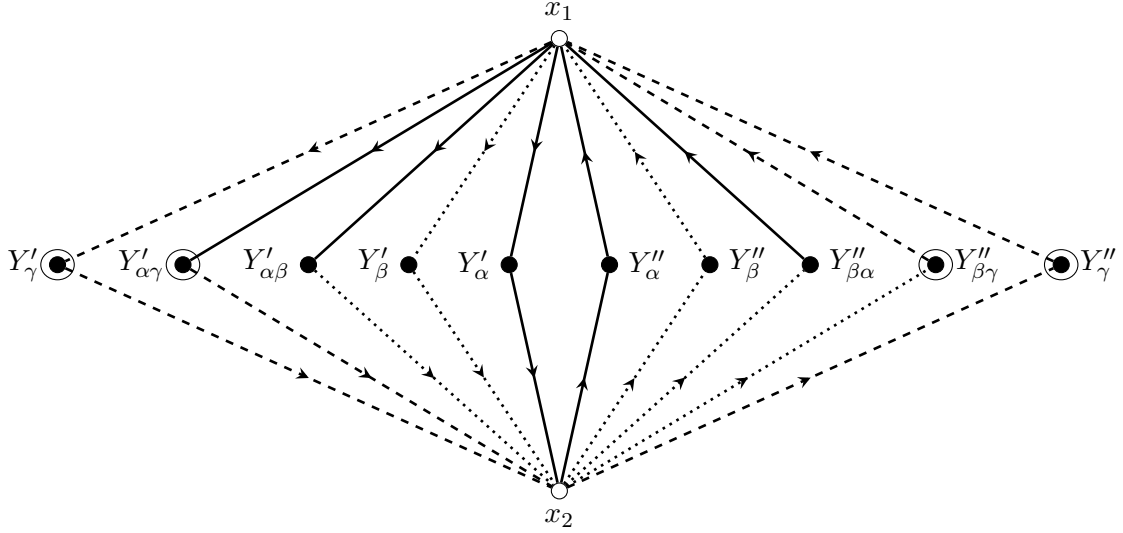
\begin{figure}[ht]
\label{figure: an arc-colored BT with no sink and no source}
\begin{center}
\begin{tikzpicture}[x=1.1cm, y=1.0cm, every edge/.style=
        {
        draw,line width=1pt,
        postaction={decorate,
                     decoration={markings,mark=at position 0.5 with {\arrow{stealth}}}
                   }
        }
]


    \draw(-6,0) circle (0.2);
    \draw(6,0) circle (0.2);
    \draw(-4.5,0) circle (0.2);
    \draw(4.5,0) circle (0.2);

    \vertex (x1) at (0,3) [label=above:{$x_{1}$}]{};
    \vertex (x2) at (0,-3) [label=below:{$x_{2}$}]{};

    \vertex (y1)[fill] at (-6,0) [label=left:{$Y'_{\gamma}$}]{};
    \vertex (y2)[fill] at (-4.5,0) [label=left:{$Y'_{\alpha\gamma}$}]{};
    \vertex (y3)[fill] at (-3.0,0) [label=left:{$Y'_{\alpha\beta}$}]{};
    \vertex (y4)[fill] at (-1.8,0) [label=left:{$Y'_{\beta}$}]{};
    \vertex (y5)[fill] at (-0.6,0) [label=left:{$Y'_{\alpha}$}]{};
    \vertex (y6)[fill] at (0.6,0) [label=right:{$Y''_{\alpha}$}]{};
    \vertex (y7)[fill] at (1.8,0) [label=right:{$Y''_{\beta}$}]{};
    \vertex (y8)[fill] at (3.0,0) [label=right:{$Y''_{\beta\alpha}$}]{};
    \vertex (y9)[fill] at (4.5,0) [label=right:{$Y''_{\beta\gamma}$}]{};
    \vertex (y10)[fill] at (6,0) [label=right:{$Y''_{\gamma}$}]{};

    \path (x1) edge[dashed] (y1);
    \path (x1) edge[] (y2);
    \path (x1) edge[] (y3);
    \path (x1) edge[dotted] (y4);
    \path (x1) edge[] (y5);

    \path (y1) edge[dashed] (x2);
    \path (y2) edge[dashed] (x2);
    \path (y3) edge[dotted] (x2);
    \path (y4) edge[dotted] (x2);
    \path (y5) edge[] (x2);

    \path (x2) edge[] (y6);
    \path (x2) edge[dotted] (y7);
    \path (x2) edge[dotted] (y8);
    \path (x2) edge[dotted] (y9);
    \path (x2) edge[dashed] (y10);

    \path (y6) edge[] (x1);
    \path (y7) edge[dotted] (x1);
    \path (y8) edge[] (x1);
    \path (y9) edge[dashed] (x1);
    \path (y10) edge[dashed] (x1);

\end{tikzpicture}
\end{center}
\caption{An arc-colored bipartite tournament with no sink and no source.}
\end{figure}

Since $D$ contains no two contractible vertices,
we have $|Y'_{\alpha}|,|Y'_{\beta}|,|Y'_{\alpha\beta}|,|Y''_{\alpha}|,|Y''_{\beta}|,|Y'_{\beta\alpha}|\leq 1$.
Note that no two vertices of $Y''_{\alpha}\cup Y''_{\beta\alpha}$ are connected by a PC path in $D$ and also the following holds.
$$
Y^{*}=Y'_{\alpha\beta}\cup Y'_{\alpha\gamma}\neq \emptyset,~Y^{**}=Y''_{\beta\alpha}\cup Y'_{\beta\gamma}\neq \emptyset,
$$
$$
Y'=Y'_{\alpha}\cup Y'_{\beta}\cup Y'_{\gamma}\cup Y'_{\alpha\beta}\cup Y'_{\alpha\gamma},~~Y''=Y''_{\alpha}\cup Y''_{\beta}\cup Y''_{\gamma}\cup Y''_{\beta\alpha}\cup Y'_{\beta\gamma}.
$$
We distinguish two subcases.

\begin{subcase}
$\alpha=\beta$.
\end{subcase}

It follows that $Y'_{\alpha\beta}=Y''_{\beta\alpha}=\emptyset$
and $Y'_{\alpha\gamma}=Y^{*}\neq \emptyset$.
If $Y''_{\alpha}=\emptyset$,
then each vertex in $Y''$ can reach $x_{2}$ by a PC path passing through $x_{1}$ and an arbitrary vertex in $Y'_{\alpha\gamma}$,
which implies that $\{x_{2}\}$ is a PCP-kernel.
If $Y''_{\alpha}\neq \emptyset$,
then since each vertex in $Y'\backslash Y'_{\alpha}$ can reach $Y''_{\alpha}$ by a PC path passing through $x_{2}$,
and each vertex in $Y''\backslash Y''_{\alpha}$ can reach $Y''_{\alpha}$ by a PC path passing through $x_{1}$ together with an arbitrary vertex in $Y'_{\alpha\gamma}$ and $x_{2}$,
together with the observation that no two vertices in $Y'_{\alpha}\cup Y''_{\alpha}$ are connected by a PC path,
we can get that $Y'_{\alpha}\cup Y''_{\alpha}$ is a PCP-kernel.

\begin{subcase}
$\alpha\neq \beta$.
\end{subcase}

If $Y''_{\alpha}=Y''_{\beta\alpha}=\emptyset$,
then since $Y'_{\alpha\beta}\cup Y'_{\alpha\gamma}\neq \emptyset$
we get that each vertex in $Y''$ can reach $x_{2}$ by a PC path passing through $x_{1}$ and an arbitrary vertex in $Y'_{\alpha\beta}\cup Y'_{\alpha\gamma}$.
It follows that $\{x_{2}\}$ is a PCP-kernel.

If $Y''_{\alpha}\neq \emptyset$ and $Y''_{\beta\alpha}\neq \emptyset$,
then each vertex in $Y'$ can reach $Y''_{\alpha}\cup Y''_{\beta\alpha}$ by a PC path passing through $x_{2}$,
and each vertex in $Y''\backslash (Y''_{\alpha}\cup Y''_{\beta\alpha})$ can reach $Y''_{\alpha}$ by a PC path passing through $x_{1}$, an arbitrary vertex in $Y'_{\alpha\beta}\cup Y'_{\alpha\gamma}$ and $x_{2}$.
Recall that no two vertices of $Y''_{\alpha}\cup Y''_{\beta\alpha}$ are connected by a PC path in $D$.
So $Y''_{\alpha}\cup Y''_{\beta\alpha}$ is a PCP-kernel.

If $Y''_{\alpha}\neq \emptyset$ and $Y''_{\beta\alpha}=\emptyset$,
then we can show that either $Y''_{\alpha}$ or $Y'_{\alpha}\cup Y''_{\alpha}$ is a PCP-kernel.
If $Y'_{\alpha}=\emptyset$,
then each vertex in $Y'$ can reach $Y''_{\alpha}$ by a PC path passing through $x_{2}$,
and each vertex in $Y''\backslash Y''_{\alpha}$ can reach $Y''_{\alpha}$ by a PC path passing through $x_{1}$, an arbitrary vertex in $Y'_{\alpha\beta}\cup Y'_{\alpha\gamma}$ and $x_{2}$.
It follows that $Y''_{\alpha}$ is a PCP-kernel.
If $Y'_{\alpha}\neq \emptyset$,
noting that no two vertices of $Y'_{\alpha}\cup Y''_{\alpha}$ are connected by a PC path,
then we can similarly show that $Y'_{\alpha}\cup Y''_{\alpha}$ is a PCP-kernel.

Now assume that $Y''_{\alpha}=\emptyset$ and $Y''_{\beta\alpha}\neq \emptyset$.
If $Y'_{\beta}=Y'_{\alpha\beta}=\emptyset$,
then $Y'_{\alpha\gamma}=Y^{*}\neq \emptyset$,
each vertex in $Y''\backslash Y''_{\beta\alpha}$ can reach $Y''_{\beta\alpha}$ by a PC path
passing through $x_{1}$, an arbitrary vertex in $Y'_{\alpha\gamma}$ and $x_{2}$,
and clearly every vertex in $Y'$ can reach $Y''_{\beta\alpha}$ by a PC path passing through $x_{2}$.
It follows that $Y''_{\beta\alpha}$ is a PCP-kernel.
If $Y'_{\beta}=\emptyset$ and $Y'_{\alpha\beta}\neq \emptyset$,
then each vertex in $Y''\backslash Y''_{\alpha\beta}$ can reach $Y'_{\alpha\beta}$ by a PC path passing through $x_{1}$
and each vertex in $Y'\backslash Y'_{\alpha\beta}$ can reach $Y''_{\beta\alpha}$ by a PC path passing through $x_{2}$.
Recall that no two vertices of $Y'_{\alpha\beta}\cup Y''_{\beta\alpha}$ are connected by a PC path.
Then $Y'_{\alpha\beta}\cup Y''_{\beta\alpha}$ is a PCP-kernel.
If $Y'_{\beta}\neq \emptyset$ and $Y'_{\alpha\beta}\neq \emptyset$,
noting that no two vertices of $Y'_{\beta}\cup Y'_{\alpha\beta}$ are connected by a PC path,
then since each vertex in $Y''$ can reach $Y'_{\beta}\cup Y'_{\alpha\beta}$ by a PC path passing through $x_{1}$
and each vertex in $Y'\backslash (Y'_{\beta}\cup Y'_{\alpha\beta})$ can reach $Y'_{\beta}\cup Y'_{\alpha\beta}$ by a PC path passing through $x_{2}$, $Y''_{\beta\alpha}$ and $x_{1}$,
we can obtain that $Y'_{\beta}\cup Y'_{\alpha\beta}$ is a PCP-kernel.
Now let $Y'_{\beta}\neq \emptyset$ and $Y'_{\alpha\beta}=\emptyset$.
If $Y''_{\beta}=\emptyset$,
then since each vertex in $Y'\backslash Y'_{\beta}$ can reach $Y'_{\beta}$ by a PC path passing through $x_{2},Y''_{\beta\alpha},x_{1}$,
and each vertex in $Y''$ can reach $Y'_{\beta}$ by a PC path passing through $x_{1}$,
we can get that $Y'_{\beta}$ is a PCP-kernel.
If $Y''_{\beta}\neq \emptyset$,
then by observing that no two vertices of $Y'_{\beta}\cup Y''_{\beta}$ are connected by a PC path
we can similarly show that $Y'_{\beta}\cup Y''_{\beta}$ is a PCP-kernel.
\end{proof}

\section{An extension}

Recall that an arc-colored digraph is rainbow
if any two arcs receive two distinct colors.
Another interesting topic deserving further consideration is the existence of a {\em kernel by rainbow paths} in an arc-colored digraph $D$,
which is defined,
similar to the definition of MP-kernels or PCP-kernels,
as a set $S$ of vertices of $D$ such that
(i) no two vertices of $S$ are connected by a rainbow path in $D$,
and (ii) every vertex outside $S$ can reach $S$ by a rainbow path in $D$.
Similar to the proof of Proposition \ref{proposition: finding a pcp-kernel is np-hard},
we can get the computational complexity of finding a kernel by rainbow paths in an arc-colored digraph.

\begin{proposition}
It is NP-hard to recognize whether an arc-colored digraph has a kernel by rainbow paths or not.
\end{proposition}

\begin{proof}
Let $D$ and $D'$ be defined as in Proposition \ref{proposition: finding a pcp-kernel is np-hard}.
Color $D'$ by using $m$ colors in such a way that the subdigraph $D$ is monochromatic
and the arc set $\{uv:~u\in V^{*}, v\in V(D)\}$ is $m$-colored.
Then one can see that the $m$-colored $D'$ has a kernel by rainbow paths if and only if $D$ has a kernel.
By Theorem \ref{thm: finding a kernel is npc} the desired result holds.
\end{proof}

\end{spacing}
\end{document}